\documentclass[12pt]{article}
\usepackage{amsmath, amssymb, amsfonts, amsthm, amscd, vmargin}
\usepackage[latin1]{inputenc}
\usepackage[all]{xy}
\usepackage{graphicx}

\setmarginsrb{2.4cm}{4.5cm}{2.4cm}{3.3cm}{0cm}{0cm}{0cm}{1.5cm}

\theoremstyle{plain}
\newtheorem{teo}{Theorem}
\newtheorem{lem}[teo]{Lemma}
\newtheorem{prop}[teo]{Proposition}

\theoremstyle{definition}
\newtheorem{defi}[teo]{Definition}

\newtheorem{rem}[teo]{Remark}

\newcommand{\Cbb}{{\mathbb C}}
\newcommand{\Qbb}{{\mathbb Q}}
\newcommand{\Zbb}{{\mathbb Z}}
\newcommand{\Rbb}{{\mathbb R}}
\newcommand{\Pbb}{{\mathbb P}}

\newcommand{\lra}{\longrightarrow}

\begin{document}

\title{The pseudo-index of horospherical Fano varieties}

\author{Boris Pasquier}
\maketitle

\begin{abstract}
We prove a conjecture of L.~Bonavero, C.~Casagrande, O.~Debarre and
S.~Druel, on the pseudo-index of smooth Fano varieties, in the
special case of horospherical varieties.
\end{abstract}



~\\

Let $X$ be a normal, complex, projective algebraic variety of
dimension $d$. Assume that $X$ is Fano, namely the anticanonical
divisor $-K_X$ is Cartier and ample. The {\it pseudo-index}
$\iota_X$ is the positive integer defined by
$$\iota_X:=\operatorname{min}\{-K_X.C\mid C\mbox{ rational curve in
}X\}.$$

The aim of this paper is to prove the following result.

\begin{teo}\label{Theo}
Let $X$ be a  $\Qbb$-factorial horospherical Fano variety of
dimension $d$, Picard number $\rho_X$ and pseudo-index $\iota_X$.
Then $$(\iota_X-1)\rho_X\leq d.$$ Moreover, equality holds if and
only if $X$ is isomorphic to $(\Pbb^{\iota_X-1})^{\rho_X}.$
\end{teo}

This inequality has been conjectured by L.~Bonavero, C.~Casagrande,
O.~Debarre and S.~Druel for all smooth Fano varieties \cite{BCDD}.
Moreover, they have proved Theorem \ref{Theo} in the case of Fano
varieties of dimension~3 and~4, all flag varieties, and also for
some particular toric varieties. More generally, this result has
been also proved by C.~Casagrande for all $\Qbb$-factorial toric
Fano varieties \cite{Ca06}. Theorem \ref{Theo} generalizes these
results to the family of horospherical varieties that contains in
particular the families of toric varieties and flag varieties.

We will use arguments similar to the ones used by C.~Casagrande, and
we will use also a few results on horospherical varieties from
\cite{Pa06}. So, let us summarize the theory of Fano horospherical
varieties (See \cite{Pa06} for more details).
~\\

{\it Horospherical} varieties are normal algebraic varieties where a
connected reductive algebraic group $G$ acts with an open orbit
isomorphic to a torus bundle over a flag variety. The dimension of
the torus is called the {\it rank} of the variety and denoted by
$n$. Toric varieties (where $G=(\Cbb^*)^n$) and flag varieties
(where $n=0$) are the first examples of horospherical varieties.

A homogeneous space $G/H$ is said to be horospherical if it is a
torus bundle over a flag variety, so that all {\it $G/H$-embeddings}
({\it i.e.} normal $G$-varieties with an open $G$-orbit isomorphic
to $G/H$) are horospherical varieties.

Let $G/H$ be a horospherical homogeneous space. Denote by $P$ the
normalizer of $H$ in $G$, it is a parabolic subgroup of $G$. There
exists a Borel subgroup $B$ of $G$ contained in $P$ such that $H$
contains the unipotent radical of $B$. Then $P$ defines a subset $I$
of the set $S$ of simple roots of $G$, so that $S\backslash I$ is
the set of simple roots $\alpha$ such that the fundamental
associated weight $\omega_\alpha$ is a character of $P=P_I$. For all
$J\subset S$, we define the associated parabolic subgroup $P_J$ of
$G$ containing $B$ as above.

Let us define a lattice $M$ as the set of characters of $P$ whose
restrictions to $H$ are trivial. Let $N$ be the dual lattice of $M$.
Let $M_\Rbb:=M\otimes_\Zbb\Rbb$ and $N_\Rbb:=N\otimes_\Zbb\Rbb$.

For all $\alpha\in S\backslash I$, we denote by $\check\alpha_M$ the
restriction of the coroot $\check\alpha$ to $M$. We also define
$a_\alpha:= \langle 2\rho^P,\check\alpha\rangle$ where $2\rho^P$ is
the sum of all the positive roots not generated by the simple roots
of $I$. (Remark that $a_\alpha\geq 2$.)

We may now formulate the classification of horospherical Fano
varieties.

\begin{defi} \label{reflexif}
Let $G/H$ be a homogeneous horospherical space.
A convex polytope $Q$ in $N_\Rbb$ is said to be  {\it $G/H$-reflexive} if the following three conditions are satisfied:\\

(1) The vertices of $Q$ are in $N\cup\{\frac{\check\alpha_M}{a_\alpha}\mid\alpha\in S\backslash I\}$, and the interior $\stackrel{o}{Q}$ of $Q$ contains 0.\\

(2) $Q^*:=\{v\in M_\Rbb\mid \forall u\in Q,\,\langle v,u\rangle\geq-1\}$ is a lattice polytope ({\it i.e.} has its vertices in $M$).\\

(3) For all $\alpha\in S\backslash I$, $\frac{\check\alpha_M}{a_\alpha}\in Q$.\\
\end{defi}

Then we have the following classification  \cite[Chapter
3]{Pa06}.

\begin{prop}
Let $G/H$ be a horospherical homogeneous space. Then the set of Fano
$G/H$-embeddings is in bijection with the set of $G/H$-reflexive
polytopes.
\end{prop}

From now on, let us fix a horospherical homogeneous space $G/H$, a
Fano $G/H$-embed\-ding  $X$ and the associated $G/H$-reflexive
polytope $Q$. We define the {\it colors} of $X$ to be the roots
$\alpha\in S\backslash I$ such that
$\frac{\check\alpha_M}{a_\alpha}\in Q\backslash \stackrel{o}{Q}$.
Denote by $\mathcal{D}_X$ the set of colors of $X$.

\begin{rem}\label{remcoloredfan}
There is a classification of $G/H$-embeddings in terms of colored
fans, due to D.~Luna and T.~Vust \cite{LV83}. The colored fan of the
Fano $G/H$-embedding $X$ associated to a $G/H$-reflexive polytope
$Q$ is the set of colored cones $(\mathcal{C}_F,\mathcal{D}_F)$ and
their colored faces, where $\mathcal{C}_F$ is the cone of $N_\Rbb$
generated by the facet $F$ of $Q$ and
$\mathcal{D}_F:=\{\alpha\in\mathcal{D}_X\mid\check\alpha_M\in\mathcal{C}_F\}$.
See \cite[Ch.1 and Ch.3]{Pa06} for details.
\end{rem}

Moreover, $X$ is $\Qbb$-factorial if and only if $Q$ is simplicial
and the points $\frac{\check\alpha_M}{a_\alpha}\in Q\backslash
\stackrel{o}{Q},\,\alpha\in\mathcal{D}_X$ are disjoint vertices of
$Q$. From now, suppose that $X$ is $\Qbb$-factorial. Then the Picard
number of $X$ is given by $\rho_X=r+\sharp (S\backslash
I)-\sharp(\mathcal{D}_X)$ where $n+r$ is the number of vertices of
$Q$.\footnote{Note that, as in the toric case, the Picard number of
$\Qbb$-factorial horospherical Fano varieties is bounded by twice
the dimension \cite[Theorem 0.2]{Pa06}.} And for all vertices $u$ of
$Q$, one can now define an integer $a_u$ to be $a_\alpha$ if there
exists $\alpha\in S\backslash I$ such that
$u=\frac{\check\alpha_M}{a_\alpha}$, and $1$ otherwise.

Let us denote by $V(Q)$ and $V(Q^*)$ the set of vertices of $Q$ and
$Q^*$ respectively. The facets of $Q$ are in bijection with the
vertices of $Q^*$ (and likewise for facets of $Q^*$). We denote by
 $F_v$ the facet of $Q$ corresponding to
 a vertex $v$ of $Q^*$. Remark that the
$F_v$ is the set of elements $u$ of $Q$ satisfying $\langle
v,u\rangle=-1$. By Luna-Vust theory, $V(Q^*)$ ({\it i.e.}, the set
of facets of $Q$) is in bijection with the set of closed $G$-orbits
of $X$. Moreover, in our case, we can describe explicitly the closed
$G$-orbits of $X$ as follows. Let $v\in V(Q^*)$ and
$J_v:=\{\alpha\in S\backslash
I\mid\frac{\check\alpha_M}{a_\alpha}\in F_v\}.$ Then the $G$-orbit
associated to $v$ is isomorphic to $G/P_{I\cup J_v}$.

Let us now define $$\epsilon_Q:=\operatorname{min}\{a_u(1+\langle
v,u \rangle)\mid u\in V(Q),\, v\in V(Q^*),\, u\not\in F_v\}.$$

The pseudo-index of $X$ is related to $\epsilon_Q$ as follows.

\begin{prop}\label{lem1}
Let $X$ and $Q$ as above. Then,

(i) $\iota_X\leq \epsilon_Q$

(ii) $\forall \alpha\in (S\backslash I)\backslash \mathcal{D}_X$,
$\iota_X\leq a_\alpha$.
\end{prop}

To prove this result we need the following lemma.

\begin{lem}\label{lem2}
Let $u\in V(Q)$ and $v\in V(Q^*)$ such that
$\epsilon_Q=a_u(1+\langle v,u \rangle)$. Then $u$ is adjacent to
$F_v$ ({\it i.e.} $F_v$ and some facet containing $u$ intersect
along a $(n-2)$-dimensional face).
\end{lem}

\begin{proof}

Let $v\in V(Q^*)$. Let $e_1,\dots,e_n$ be the vertices of the
complex $F_v$. For all $j\in\{1,\dots,n\}$, denote by $F_j$ the
facet of $Q$ containing $e_1,\dots,e_{j-1},e_{j+1},\dots,e_n$
distinct from $F_v$ and denote by $u^j$ the vertex of $F_j$ distinct
from $e_1,\dots,e_n$.

Let $(e_1^*,\dots,e_n^*)$ be the dual basis of $(e_1,\dots,e_n)$ in
$M_\Rbb$. Let $j\in\{1,\dots,n\}$. Then $\langle
e_j^*,u^j\rangle\neq 0$ (otherwise, $u^j$ is in the hyperplane
generated by the $(e_i)_{i\neq j}$) and one can define
$$\gamma_j=\frac{-1-\langle v,u^j\rangle}{\langle
e_j^*,u^j\rangle}.$$ Moreover, the vertex of $Q^*$ associated to
$F_j$ is $v^j=v+\gamma_je_j^*$. Then $\gamma_j>0$ because $\langle
v^j,e_j\rangle >-1$.

Let $v\in V(Q^*)$ and $u\in V(Q)$ such that
$\epsilon_Q=a_u(1+\langle v,u \rangle)$. Then $$a_u(1+\langle
v^j,u\rangle)=a_u(1+\langle v,u\rangle)+ a_u\gamma_j\langle
e^*_j,u\rangle,$$ so that
$$\langle e_j^*,u\rangle<0 \Longleftrightarrow \langle v^j,u\rangle=-1 \Longleftrightarrow u\in F_j,$$
or, in other words:
$$u\not\in F_j \Longleftrightarrow \langle e_j^*,u\rangle\geq 0.$$

Now, if  $u\neq u^j$ for all $j\in\{1,\dots,n\}$, then  $\langle e_j^*,u\rangle\geq 0$ for all $j\in\{1,\dots,n\}$ and then $u$ is in the cone generated by
 $e_1,\dots,e_n$, that is a contradiction. So, $u$ is one of the $u^j$, {\it i.e.}, $u$ is adjacent to $F_v$.
\end{proof}

We need also to define two different families of rational curves in
$X$. The first family is very similar to the family of the curves
stable under the action of the torus in the case of toric varieties.
And the second one is a family of Schubert curves in  the closed
$G$-orbits of $X$.

Denote by $x_0$ the point of $X$ with isotropy group $H$. Remark
that $P/H\simeq(\Cbb^*)^n$. Then the closure $X'$ of the $P$-orbit
of $x_0$ in $X$ is a toric variety associated to the fan consisting
of the cones $(\mathcal{C}_{F_v})_{v\in Q^*}$ (defined in Remark
\ref{remcoloredfan}) and their faces, {\it i.e.} the fan of $X$
without colors.

Let $\mu$ be a $(n-2)$-dimensional face of $Q$ (or, that is the same
a $(n-1)$-dimensional colored cone of the colored fan of $X$, or
also a $(n-1)$-dimensional cone of the fan of $X'$). Then there
exists a unique $P$-stable curve of $X'\subset X$ containing two
$P$-fixed points lying in the two closed $G$-orbits of $X$
corresponding to the two facets of $Q$ containing $\mu$. This curve
is rational; we denoted it by $C_\mu$.

Let $v\in V(Q^*)$ and $\alpha\in(S\backslash
I)\backslash\mathcal{D}_X$ such that
$\check\alpha_M\in\mathcal{C}_{F_v}$. Recall that $v$ corresponds to
a closed $G$-orbit $Y$ of $X$. Moreover, since $\alpha$ is not a
color of $X$, $Y=G/P_Y$, where $P_Y$ is a parabolic subgroup of $G$
containing $B$ such that $\omega_\alpha$ is a character of $P_Y$.
Then we can define the rational $B$-stable curve $C_{\alpha,v}$ in
$X$ to be the Schubert curve $\overline{Bs_\alpha P_Y/P_Y}$ where
$s_\alpha$ is the simple reflection associated to the simple root
$\alpha$.

It is not difficult to check that the curve $C_\mu$ is the one
defined in \cite[Prop.3.3]{Br93}. Moreover, the curve $C_{\alpha,v}$
is the one defined in \cite[Prop.3.6]{Br93}. Indeed, $C_{\alpha,v}$
is contracted by the extremal contraction described in
\cite[Ch.3.4]{Br93}, that consists in ''adding'' the color $\alpha$.
In fact, the curves $C_\mu$ and $C_{\alpha,v}$ are exactly the
irreducible $B$-stable curves in $X$ and they generate the cone of
effective 1-cycles.

Now, using the definition of the polytope $Q$ given in
\cite[Ch.3]{Pa06} and the intersection formulas given in
\cite[Ch.3.2]{Br93}, one can prove the following lemma.

\begin{lem}\label{lemnoir}
\begin{enumerate}
\item  Let $\mu$ a $(n-2)$-dimensional face of $Q$, choose $v\in
V(Q^*)$ and $u\in V(Q)$ such that $F_v$ and some facet of $Q$
containing $u$ intersect along $\mu$. Let $\chi_\mu$ the primitive
element of $M$ such that $\langle\chi_\mu,x\rangle=0$ for all
$x\in\mu$ and $\langle\chi_\mu,u\rangle>0$. Then we have
$$-K_X.C_\mu =\frac{a_u(1+\langle v,
u\rangle)}{\langle\chi_\mu,a_uu\rangle}.$$

\item Let $v\in V(Q^*)$ and $\alpha\in (S\backslash
I)\backslash\mathcal{D}_X$ such that
$\check\alpha_M\in\mathcal{C}_{F_v}$. Then we have
$$-K_X.C_{\alpha,v} =a_\alpha+\langle v,
\check\alpha_M\rangle.$$
\end{enumerate}
\end{lem}

We are now able to prove easily Proposition \ref{lem1}.

\begin{proof}[Proof of Proposition \ref{lem1}]
(i) Let $u\in V(Q)$ and $v\in V(Q^*)$ such that
$\epsilon_Q=a_u(1+\langle v,u \rangle)$, then by Lemma \ref{lem2},
$u$ is adjacent to $F_v$. This means that $u$ lies in a facet $F$ of
$Q$ such that $F$ and $F_v$ intersect along a $(n-2)$-dimensional
face $\mu$. Then, by Lemma \ref{lemnoir}, we have
$$-K_X.C_\mu =\frac{a_u(1+\langle v, u\rangle)}{\langle\chi_\mu,a_uu\rangle},$$
where $\chi_\mu\in M$ and $\langle\chi_\mu,a_uu\rangle$ is a
positive integer. Then we have $$\iota_X\leq -K_X.C_\mu \leq
a_u(1+\langle v, u\rangle)=\epsilon_Q.$$

(ii) Let $\alpha\in (S\backslash I)\backslash\mathcal{D}_X$. Let us
choose $v\in V(Q^*)$ such that $\check\alpha_M\in\mathcal{C}_{F_v}$.
Then, by Lemma~\ref{lemnoir},
$$\iota_X\leq -K_X.C_{\alpha,v}=a_\alpha+\langle v,
\check\alpha_M\rangle \leq a_\alpha$$ because $\langle v,
\check\alpha_M\rangle$ is non-positive by the choice of $v$.
\end{proof}

\begin{rem}\label{remnoncolor}
In (ii), if $\check\alpha_M\neq 0$, we have in fact the strict
inequality $\iota_X<a_\alpha$.
\end{rem}

\begin{proof}[Proof of the inequality of Theorem \ref{Theo}]
Since the origin lies in the interior of $Q^*$, there exists a
relation \begin{equation}\label{eq1} m_1v_1+\cdots
+m_hv_h=0\end{equation} where $h>0$, $v_1,\dots,v_h$ are vertices of
$Q^*$, and $m_1,\dots,m_h$ are positive integers. Let
$I:=\{1,\dots,h\}$ and $M:=\sum_{i\in I}m_i$. For any $u\in V(Q)$,
we define
$$A(u)=\{i\in I\mid\langle v_i,u\rangle=-1\}.$$

 Now, we have
 \begin{eqnarray}
0 & = & \sum_{i\in I}m_i\langle v_i,u\rangle \nonumber\\
  & = & -\sum_{i\in A(u)}m_i+\sum_{i\not\in A(u)}m_i\langle
  v_i,u\rangle \nonumber\\
  & \geq & -\sum_{i\in A(u)}m_i+\sum_{i\not\in A(u)}m_i(\frac{\epsilon_Q}{a_u}-1)=(\frac{\epsilon_Q}{a_u}-1)M-\frac{\epsilon_Q}{a_u}\sum_{i\in
  A(u)}m_i\label{eq2}
\end{eqnarray}
so that $$\frac{\epsilon_Q-a_u}{\epsilon_Q}M\leq \sum_{i\in
A(u)}m_i.$$

Denote by $r$ the positive integer such that the number of vertices
of $Q$ is $n+r$. Let us sum the preceding inequalities over all
vertices $u$ of $Q$. We obtain
$$(n+r)M-\frac{\sum_{u\in V(Q)}a_u}{\epsilon_Q}M\leq\sum_{u\in
V(Q)}\sum_{i\in A(u)}m_i=nM\mbox{ and }\epsilon_Qr\leq\sum_{u\in
V(Q)}a_u.$$

(The last equality comes from the fact that $X$ is $\Qbb$-factorial,
so that each facet of $Q$ has exactly $n$ vertices.)

Then, by Proposition \ref{lem1} (i), we have
\begin{eqnarray}\iota_Xr & \leq & \epsilon_Q r\label{eq3}\\ & \leq & \sum_{u\in V(Q)}a_u.\nonumber\end{eqnarray} But recall
that $\rho_X=r+\sharp(S\backslash I)-\sharp(\mathcal{D}_X)$ and then
$$\sum_{u\in V(Q)}a_u=\sum_{\alpha\in\mathcal{D}_X}a_\alpha + \rho_X
+ n- \sharp(S\backslash I).$$ Thus, we have
$$\iota_X(\rho_X-\sharp(S\backslash I)+\sharp(\mathcal{D}_X))\leq
\sum_{\alpha\in\mathcal{D}_X}a_\alpha + \rho_X + n-
\sharp(S\backslash I)$$ and, by Proposition \ref{lem1}, we obtain
that
\begin{eqnarray}
\rho_X(\iota_X-1) & \leq &
\sum_{\alpha\in\mathcal{D}_X}a_\alpha+\sum_{\alpha\in (S\backslash
I)\backslash\mathcal{D}_X}\iota_X -\sharp(S\backslash I) +n\nonumber\\
 & \leq & \sum_{\alpha\in\mathcal{D}_X}a_\alpha+\sum_{\alpha\in (S\backslash
I)\backslash\mathcal{D}_X}a_\alpha -\sharp(S\backslash I) +n =
\sum_{\alpha\in S\backslash I}(a_\alpha-1)+n.\label{eq4}
\end{eqnarray}
This yields the inequality of Theorem \ref{Theo} by \cite[Lemmma
4.13]{Pa06} that asserts that
\begin{equation}\label{eq5}\sum_{\alpha\in S\backslash
I}(a_\alpha-1)+n\leq d. \end{equation}
\end{proof}

Let us study now the cases of equality. In view of the above proof,
one can remark that $(\iota_X-1)\rho_X=d$ if and only if equality
holds in inequalities \ref{eq2}, \ref{eq3}, \ref{eq4} and \ref{eq5},
{\it i.e.} all the following conditions are satisfied:
\begin{enumerate}
\item $\forall u\in V(Q)$, $\forall i\in\{1,\dots,h\}$,
$a_u(1+\langle v_i,u\rangle)$ equals either 0 or $\epsilon_Q$;
\item $\iota_X=\epsilon_Q$;
\item $\forall\alpha\in(S\backslash
I)\backslash\mathcal{D}_X$, $\iota_X=a_\alpha$;
\item $\sum_{\alpha\in S\backslash
I}(a_\alpha-1)+n=d$.
\end{enumerate}

Let us make a few remarks about these conditions.

One can take a relation as \ref{eq1} involving all vertices of
$Q^*$. Thus, Condition 1 can be replaced by
\begin{description}
\item 1'. $\forall u\in V(Q)$, $\forall v\in V(Q^*)$, $a_u(1+\langle
v,u\rangle)$ equals either 0 or $\epsilon_Q$.
\end{description}

\begin{lem}\label{lemmacondition2}
Conditions 2 implies that $X$ is locally factorial.
\end{lem}

\begin{proof}
Let us suppose that $X$ is not locally factorial (but
$\Qbb$-factorial). Then there exists a facet of $Q$ with vertices
$u_1,\dots,u_n$ such that $(a_{u_1}u_1,\dots,a_{u_n}u_n)$ is a basis
of $N_\Rbb$ but not of $N$ (see \cite[Proposition 4.4]{Pa06}). This
means that there exist a non-zero element
$u_0=\sum_{i=1}^n\lambda_ia_{u_i}u_i$ of $N$ such that $\forall
i\in\{1,\dots,n\}$, $\lambda_i\in \Qbb$, $0\leq\lambda_i\leq 1$ and
some $i$ satisfying $0<\lambda_i<1$. Let $\mu$ be the wall generated
by $u_1,\dots,u_{j-1},u_{j+1},\dots,u_n$ and let $\chi_\mu$ be the
associated element of $M$ as in the proof of Proposition \ref{lem1}
(i) (with $\chi_\mu(u_j)>0$). Then
$\chi_\mu(a_{u_j}u_j)=\frac{1}{\lambda_j}\chi_\mu(u_0)>1$ because
$\chi_\mu(u_0)\in\Zbb_{>0}$. This implies, by the proof of
Proposition \ref{lem1} (i) that $\iota_X<\epsilon_Q$.
\end{proof}

By Remark \ref{remnoncolor}, Condition 3 implies that every simple
root $\alpha$ such that $\check\alpha_M\neq 0$ is a color of $X$.

Using the proof of \cite[Lemmma 4.13]{Pa06}, one can prove that
Condition 4 is equivalent to
\begin{description}
\item
4'. The decomposition $\Gamma=\sqcup_{i=1}^\gamma\Gamma_i$ into
connected components of the Dynkin diagram~$\Gamma$ of $G$
satisfies: $\forall i\in\{1,\dots,\gamma\}$, either no vertex of
$\Gamma_i$ corresponds to a simple root of $S\backslash I$ or only
one vertex of $\Gamma_i$ corresponds to a simple root $\alpha_i$ of
$S\backslash I$. In the latter case, $\Gamma_i$ is of type $A_m$ or
$C_m$, and $\alpha_i$ corresponds to a simple end of $\Gamma_i$
("simple" means not adjacent to a double edge).\end{description}

Remark that Conditions 2 and 4' imply, together with Lemma
\ref{lemmacondition2}, that $X$ is smooth (by the smoothness
criterion of horospherical varieties \cite[Theorem 2.6]{Pa06}).

We may now complete the proof of Theorem \ref{Theo}.

\begin{proof}[Proof of the case of equality of Theorem \ref{Theo}]
Step 1: one can suppose that $G$ is the direct product of a
semi-simple group $G'$ and a torus (for example, one can replace $G$
by the product of the semi-simple part of $G$ by the torus $P/H$,
see \cite[proof of Proposition 3.10]{Pa06}). Then one can also
suppose that $G'$ is a direct product of simple groups. Then,
because of Condition~4', we have $$X=X'\times\prod_{\alpha\in
S\backslash I, \check\alpha_M=0}G/P(\omega_\alpha),$$ where $X'$ is
a horospherical variety. Moreover, for all $\alpha\in S\backslash I$
such that $\check\alpha_M=0$, the variety $G/P(\omega_\alpha)$ is
isomorphic to $\Pbb^{a_\alpha-1}$ by condition 4', and hence to
$\Pbb^{\iota_X-1}$ by Condition 3. Also, $X'$ has the same
combinatorial data as $X$, except for the colors with zero image.

Then we have to prove the result in the case where there is no
$\alpha\in S\backslash I$ such that $\check\alpha_M=0$. In that
case, we have $\mathcal{D}_X=S\backslash I$.\\

Step 2: Let $v$ be a vertex of $Q^*$ and let $e_1,\dots,e_n$ the
vertices of $Q$ satisfying $\langle v,e_i\rangle=-1$. Let us denote
by $f_1,\dots,f_r$ the remaining vertices of $Q$. Let
$K:=\{1,\dots,n\}$ and $J:=\{1,\dots,r\}$. For any $k\in K$, the
face of $Q$ with vertices $e_1,\dots,e_{k-1},e_{k+1},\dots,e_n$ lies
on exactly two facets, one of which has vertices $e_1,\dots,e_n$.
Thus, there exists a unique $\phi(k)\in J$ such that
$f_{\phi(k)},e_1,\dots,e_{k-1},e_{k+1},\dots,e_n$ are the vertices
of a facet $F_k$ of $Q$. This defines a function $\phi:K\lra J$.

Let $(e_1^*,\dots,e_n^*)$ the dual basis of $(e_1,\dots,e_n)$ in
$M_\Rbb$. We have $v=-e_1^*-\cdots -e_n^*$. Fix $k\in K$ and let
$v_k$ be the vertex of $Q^*$ corresponding to the facet $F_k$ of
$Q$. We have, by Condition 1', $$\forall i\in
K\backslash\{k\},\,a_{e_i}(1+\langle v_k, e_i\rangle)=0\mbox{ and
}a_{e_k}(1+\langle v_k, e_k\rangle)=\epsilon_Q,$$ so that
$v_k=v+\frac{\epsilon_Q}{a_{e_k}}e_k^*.$

Now for any $j\in J$ we have
$$\langle e_k^*,f_j\rangle=\frac{a_{e_k}}{\epsilon_Q}\langle
v_k-v,f_j\rangle=\left\{
\begin{array}{cc}
-\frac{a_{e_k}}{a_{f_j}} & \mbox{ if }\phi(k)=j,\\
0 & \mbox{ otherwise.}
\end{array} \right. $$
And then $$a_{f_j}f_j+\sum_{k\in\phi^{-1}(j)}a_{e_k}e_k=0.$$ One can
easily check that there exists a decomposition $M=\bigoplus_{j\in J}
M_j$ such that for all $j\in J$,
$$M_j=\{m\in M\mid\langle m,e_k\rangle=0,\,\forall k\in
K\backslash\phi^{-1}(j)\}.$$ We can also assume that the set
$S\backslash I$ is in bijection with the set of normal simple
subgroups of $G$, because of Step 1 and Condition 4'. So that one
can decompose $G$ as a product $G_1\times\dots\times G_r$ such that
for all $j\in J$, $M_j=M\cap \Lambda_{G_j}$ where $\Lambda_{G_j}$ is
the character group of a maximal torus of $G_j$. Then $G/H$ is
isomorphic to $G_1/H_1\times\dots\times G_r/H_r$ for some
horospherical subgroups $H_j$ of $G_j$. Moreover it is not difficult
to check (comparing the associated $G/H$-reflexive polytopes) that
$X$ is isomorphic to the product $X_1\times\dots\times X_r$ where
$X_j$ is the $G_j/H_j$-embedding associated to the reflexive
polytope defined as the convex hull of $f_j$ and all $e_k$ such that
$\phi(k)=j$. For all $i\in\{1,\dots,r\}$, we can compute that the
Picard number of $X_i$ is~1. Then, by Condition 4' and \cite[Theorem
1.4]{Pa08}, $X_i$ is a projective space. This completes the proof of
Theorem \ref{Theo}.
\end{proof}

\end{document}